\documentclass[10pt,reqno]{amsart}
   \usepackage[centertags]{amsmath}
   \usepackage{amsfonts}
   \usepackage{amsmath}
   \usepackage{amssymb}
   \usepackage{amsthm}
   \usepackage{newlfont}
   \usepackage[latin1]{inputenc}
\usepackage{multirow, bigdelim}

\usepackage{epsfig}
\usepackage{pst-all}
\usepackage{pstricks}
\usepackage{pgf}
\usepackage{mathrsfs}
\usepackage{hyperref}
\usepackage{bm}

\usepackage{graphicx}

\usepackage{bm}

\usepackage{mathtools}
\usepackage{ifmtarg}

\makeatletter

    \newcommand\contFrac{\@ifstar{\@contFracStar}{\@contFracNoStar}}

    \def\singleContFrac#1#2{%
        \begin{array}{@{}c@{}}%
            \multicolumn{1}{c|}{#1}%
            \\%
            \hline%
            \multicolumn{1}{|c}{#2}%
        \end{array}%
    }

    \def\@contFracNoStar#1{%
        \mathchoice{
            \@contFracNoStarDisplay@#1//\@nil%
        }{
            \@contFracNoStarInline@#1//\@nil%
        }{
            \@contFracNoStarInline@#1//\@nil%
        }{
            \@contFracNoStarInline@#1//\@nil%
        }%
    }

    \def\@contFracNoStarDisplay@#1//#2\@nil{%
        \@ifmtarg{#2}{%
            #1%
        }{%
            #1+\cfrac{1}{\@contFracNoStarDisplay@#2\@nil}%
        }%
    }

        \def\@contFracNoStarInline@#1//#2\@nil{%
            \@ifmtarg{#2}{%
                #1%
            }{%
                #1 \@@contFracNoStarInline@@#2\@nil%
            }%
        }
        \def\@@contFracNoStarInline@@#1//#2\@nil{%
            \@ifmtarg{#2}{%
                + \singleContFrac{1}{#1}%
            }{%
                + \singleContFrac{1}{#1} \@@contFracNoStarInline@@#2\@nil%
            }%
        }

    \def\@contFracStar#1{%
        \mathchoice{
            \@contFracStarDisplay@#1////\@nil%
        }{
            \@contFracStarInline@#1//\@nil%
        }{
            \@contFracStarInline@#1//\@nil%
        }{
            \@contFracStarInline@#1//\@nil%
        }%
    }

    \def\@contFracStarDisplay@#1//#2//#3\@nil{%
        \@ifmtarg{#2}{%
            #1%
        }{%
            #1 + \cfrac{#2}{\@contFracStarDisplay@#3\@nil}%
        }%
    }

        \def\@contFracStarInline@#1//#2\@nil{%
            \@ifmtarg{#2}{%
                #1%
            }{%
                #1 \@@contFracStarInline@@#2\@nil%
            }%
        }
        \def\@@contFracStarInline@@#1//#2//#3\@nil{%
            \@ifmtarg{#3}{%
                - \singleContFrac{#1}{#2}%
            }{%
                - \singleContFrac{#1}{#2} \@@contFracStarInline@@#3\@nil%
            }%
        }
\makeatother

\allowdisplaybreaks[3]

\theoremstyle{plain}
\newtheorem{theorem}{Theorem}[section]
\newtheorem{lemma}[theorem]{Lemma}
\newtheorem{proposition}[theorem]{Proposition}
\newtheorem{corollary}[theorem]{Corollary}

\theoremstyle{definition}

\theoremstyle{remark}
\newtheorem{remark}[theorem]{Remark}
\newtheorem*{remark*}{Remark}

\numberwithin{equation}{section}

\newcommand\D{\displaystyle}

\title[Stochastic LU, Darboux transforms and urn models]{Stochastic LU factorizations, Darboux transformations and urn models}

\author{F. Alberto Gr\"unbaum}
\address{F. Alberto Gr\"unbaum\\
Department of Mathematics. University of California, Berkeley. Berkeley, CA 94720  U.S.A.}
\email{grunbaum@math.berkeley.edu}
\date{\today}
\thanks{
}

\author{Manuel D. de la Iglesia}
\address{Manuel D. de la Iglesia\\
Instituto de Matem\'aticas, Universidad Nacional Aut\'onoma de M\'exico, Circuito Exterior, C.U., 04510, Ciudad de México, México.}
\email{mdi29@im.unam.mx}
\thanks{The work of both authors is supported by UC MEXUS-CONACYT grant CN-16-84. The work of the first author is also supported by PAPIIT-DGAPA-UNAM grant IA102617 (México), MTM2015-65888-C4-1-P (Ministerio de Economía y Competitividad, Spain) and FQM-262, FQM-7276 (Junta de Andaluc\'ia).}

\date{\today}



\begin{document}

\maketitle

\begin{abstract}
We consider UL (and LU) decompositions of the one-step transition probability matrix of a random walk with state space the nonnegative integers, with the condition that both upper and lower triangular matrices in the factorization are also stochastic matrices. We give conditions on the free parameter of the UL factorization in terms of certain continued fraction such that this stochastic factorization is possible. By inverting the order of multiplication (also known as a Darboux transformation) we get a new family of random walks where it is possible the identify the spectral measures in terms of a Geronimus transformation. The same can be done for the LU factorization but now without a free parameter. Finally, we apply our results in two examples, the random walk with constant transition probabilities and the random walk generated by the Jacobi orthogonal polynomials. In both situations we give urn models associated with all the random walks in question.

\end{abstract}

\section{Introduction}

The main motivation of this paper was to find simple ways to describe certain random walks on the nonnegative integers in terms of urn models. The starting point was a paper written by one the authors (see \cite{G3}) where an urn model associated with Jacobi orthogonal polynomials was given. This was a rather contrived model compared to more familiar ones such as those of Ehrenfest and Bernoulli-Laplace, and we wondered if there was an alternative way of setting a less elaborate urn model.

The way we approach our goal in this paper is to divide the urn model associated with the random walk into two different and simpler urn experiments and combine them together to obtain a simpler description of the original urn model. For that we propose to factorize the one-step transition probability matrix $P$  of the random walk (which is a tridiagonal matrix, see \eqref{PP}) into its UL decomposition, i.e. $P=P_UP_L$. Here $P_L$ and $P_U$ are lower and upper triangular matrices (bidiagonal), respectively, and we want them to be \emph{stochastic}. As it is very well known this factorization is not unique and it comes with a \emph{free parameter}. The factorization, if it can be achieved in terms of stochastic factors, will represent a family of factorizations of the original transition matrix $P$. From a probabilistic point of view the meaning of this factorization will be the composition of a family of two random walks (one pure birth and another pure death), first the one associated with $P_U$ followed by the one associated with $P_L$. We will study the conditions on the free parameter mentioned above under which this stochastic UL factorization is possible and will relate it with the theory of continued fractions. We will see that this kind of stochastic factorizations will not always be possible for certain random walks. For instance, the symmetric random walk with constant transition probabilities will give one of these bad cases. The same can be done if we consider the LU decomposition but now the difference is that the stochastic factorization (if possible) is unique and does not come with a free parameter, as the UL decomposition. The stochastic LU decomposition will now be possible under certain boundary conditions on one of the transition probabilities of the random walk, also related with a continued fraction. This will be the content of Section \ref{sec2}. We mention that our factorization is in the same spirit of one that was exhibited in a much more elaborate problem involving matrix-valued orthogonal polynomials (see \cite{GPT3}). For the use of this kind of orthogonal polynomials in the study of stochastic processes, see \cite{DRSZ, G2}.

UL and LU decompositions of stochastic matrices have been considered earlier in the literature (see for instance \cite{Gr1, Gr2, Hey, Vig}). In \cite{Hey} a decomposition is found as $I-P=(A-I)(B-S)$, where $P$ is a stochastic matrix, $A$ is strictly upper triangular, $B$ is strictly lower triangular and $S$ is diagonal. The entries of $A$ (expected values) and the entries of $B$ and $S$ (probabilities) are related with certain Markov chains also known and censored Markov chains. In \cite{Vig} UL and LU factorizations are also considered of the form $I-P=(I-L)(I-K)$, where $L$ is upper triangular and $K$ is lower triangular. That means that $P=L+K-LK$. In this paper the author explores the relation between this factorization and the very well known Wiener-Hopf factorization for Markov chains, this last one as a particular case of the LU factorizations, up to Fourier transforms (see \cite{Vig} for more details). These factorizations seem to be different from the one we try to consider here of the form $P=P_UP_L$, where the three matrices involved are stochastic.

One of the main advantages of considering a factorization of the form $P=P_UP_L$ for tridiagonal stochastic matrices $P$ is that we can make use of the so-called \emph{discrete Darboux transformation}, consisting of inverting the order of multiplication. The new matrix $\widetilde P=P_LP_U$ will also be tridiagonal and stochastic. Since the factorization comes with a free parameter, we will have a family of new random walks different in general from the original one. These Darboux transformations have been studied before (see for instance \cite{GH, GHH, SZ, Yo, Ze}) in the context of the theory of orthogonal polynomials, in particular in the description of some families of Krall polynomials. It has played an important role in the study of integrable systems (see \cite{MS}). An important property of the Darboux transformation is that one knows how to relate the spectral measure associated with $P$ with the spectral measure associated with $\widetilde P$ using the so-called Geronimus transformation. For the LU factorization we have that these measures are related through the Christoffel transformation (see Section \ref{sec3}). A more general way of performing a discrete Darboux transformation consists of introducing a new parameter $\lambda\in\mathbb{C}$ such that we find a factorization of the form $P=AB+\lambda I$, where $A$ is upper triangular and $B$ is lower triangular (or $P=BA+\lambda I$ for the LU factorization). The case we study here is $\lambda=0$ and $A$ and $B$ stochastic matrices. 

Once we have the spectral measure it is easy to analyze the corresponding random walk in terms of the orthogonal polynomials associated with that measure. This was first done by a series of papers of S. Karlin and J. McGregor (inspired by work by W. Feller and H.P. McKean) in the 1950s (see \cite{KMc2, KMc3, KMc6}) where they studied first continuous-time birth-and-death processes and then the case of discrete-time random walks. After that many other authors have been working in this connection, such as E. van Doorn, M. Ismail, J. Letessier, G. Valent and H. Dette, to mention just a few. With the knowledge of the spectral measure and the corresponding orthogonal polynomials it is possible to give an explicit integral representation of the entries of the $n$-step transition probability matrix of the random walk (Karlin-McGregor formula), and to study the corresponding invariant measure vector (or distribution if the random walk is ergodic) in terms of the inverse of the norms of the orthogonal polynomials, and to study some other probabilistic properties like recurrence, absorbing times, first return times or limit theorems.

We will apply the results of this paper to two examples in Sections \ref{sec4} and \ref{sec5}. The first one is the random walk with constant transition probabilities and the second one is the random walk generated by Jacobi orthogonal polynomials. In both cases we show how to choose the free parameter of the UL factorization (or conditions for the LU factorization) such that we obtain a family of stochastic factorizations. In some cases we compute explicitly the coefficients of both stochastic factors. We also give the spectral measure associated with the random walks and explore the corresponding Darboux transformation, as well as other probabilistic properties. Finally we give urn models for some particular cases of both examples, exploring the probabilistic meaning of the free parameter in the UL factorization. These simplified urn models arising from the stochastic factorizations are among the main results in this paper.


\section{Stochastic LU and UL factorizations}\label{sec2}

Let $P$ be the transition probability matrix of an irreducible random walk with space state $\mathbb{Z}_{\geq0}$, given by
\begin{equation}
P=\begin{pmatrix}\label{PP}
b_0&a_0&0&0\\
c_1&b_1&a_1&0&\\
0&c_2&b_2&a_2&\\
&&\ddots&\ddots&\ddots
\end{pmatrix}.
\end{equation}
\vspace{-0.1cm}
As usual, since $P$ is stochastic, we have that all entries are nonnegative and
$$
b_0+a_0=1,\quad c_n+b_n+a_n=1,\quad n\geq1.
$$
The irreducibility conditions force us to take $0<a_n,c_{n+1}<1, n\geq0$. A diagram of the transitions between the states is given by

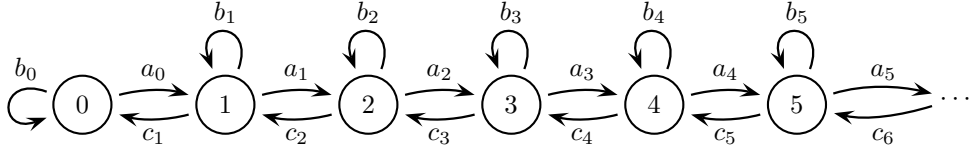
\vspace{0.7cm}

\begin{center}
$$\begin{psmatrix}[colsep=1.9cm]
  \cnode{.4}{0}& \cnode{.4}{1} & \cnode{.4}{2}& \cnode{.4}{3}& \cnode{.4}{4}& \cnode{.4}{5} &  \rnode{6}{\Huge{\cdots}} \\
\psset{nodesep=3pt,arcangle=15,labelsep=2ex,linewidth=0.3mm,arrows=->,arrowsize=1mm
3} \nccurve[angleA=160,angleB=200,ncurv=4]{0}{0}
\uput[u](0.95,1.95){a_0}\uput[u](2.85,1.95){a_1}\uput[u](4.75,1.95){a_2}\uput[u](6.65,1.95){a_3}\uput[u](8.55,1.95){a_4}\uput[u](10.65,1.95){a_5}
\uput[u](0.95,1.1){c_1}\uput[u](2.85,1.1){c_2}\uput[u](4.75,1.1){c_3}\uput[u](6.65,1.1){c_4}\uput[u](8.55,1.1){c_5}\uput[u](10.65,1.1){c_6}
\uput[u](-0.75,1.95){b_0}\uput[u](1.9,2.7){b_1}\uput[u](3.8,2.7){b_2}\uput[u](5.7,2.7){b_3}\uput[u](7.6,2.7){b_4}\uput[u](9.5,2.7){b_5}
\nccurve[angleA=70,angleB=110,ncurv=4]{1}{1}
\nccurve[angleA=70,angleB=110,ncurv=4]{2}{2}
\nccurve[angleA=70,angleB=110,ncurv=4]{3}{3}
\nccurve[angleA=70,angleB=110,ncurv=4]{4}{4}
\nccurve[angleA=70,angleB=110,ncurv=4]{5}{5}
 \ncarc{0}{1}\ncarc{1}{0} \ncarc{1}{2} \ncarc{2}{1}\ncarc{2}{3} \ncarc{3}{2}
\ncarc{3}{4} \ncarc{4}{3} \ncarc{4}{5}\ncarc{5}{4} \ncarc{5}{6} \ncarc{6}{5}
\psset{labelsep=-3.40ex}\nput{90}{0}{0}
\psset{labelsep=-3.40ex}\nput{90}{1}{1}
\psset{labelsep=-3.40ex}\nput{90}{2}{2}
\psset{labelsep=-3.40ex}\nput{90}{3}{3}
\psset{labelsep=-3.40ex}\nput{90}{4}{4}
\psset{labelsep=-3.40ex}\nput{90}{5}{5}
\end{psmatrix}
$$
\end{center}
\vspace{-1.2cm}

We would like to perform a UL decomposition of the matrix $P$ in the following way
\begin{equation}\label{Pxy}
P=\begin{pmatrix}
b_0&a_0&\\
c_1&b_1&a_1&\\
&\ddots&\ddots&\ddots
\end{pmatrix}=\begin{pmatrix}
y_0&x_0&\\
0&y_1&x_1&\\
&\ddots&\ddots&\ddots
\end{pmatrix}\begin{pmatrix}
s_0&0&\\
r_1&s_1&0&\\
&\ddots&\ddots&\ddots
\end{pmatrix}=P_UP_L,
\end{equation}
with the condition that $P_U$ and $P_L$ \emph{are also stochastic matrices}. This means that all entries of $P_U$ and $P_L$ are nonnegative and
$$
x_n+y_n=1,\quad n\geq0,\quad s_0=1,\quad r_n+s_n=1,\quad n\geq1.
$$
A direct computation shows that
\begin{align}
\nonumber a_n&=x_ns_{n+1},\quad n\geq0\\
\label{ULd}b_n&=x_nr_{n+1}+y_ns_n,\quad n\geq0,\\
\nonumber c_n&=y_nr_n,\quad n\geq1.
\end{align}
The irreducibility conditions force us to take $x_n,s_{n+1}>0, n\geq0$ and $y_n,r_{n}>0, n\geq1$. The free parameter $y_0$ satisfies $0\leq y_0<1$. Observe that $P_U$ is a pure birth random walk on $\mathbb{Z}_{\geq0}$ with diagram
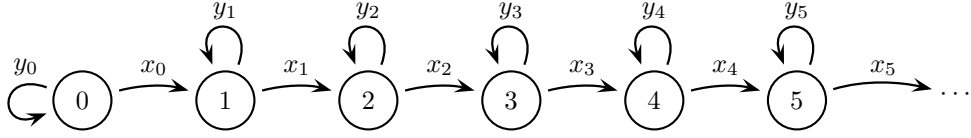
\vspace{0.7cm}
\begin{center}
$$
\begin{psmatrix}[colsep=1.9cm]
  \cnode{.4}{0}& \cnode{.4}{1} & \cnode{.4}{2}& \cnode{.4}{3}& \cnode{.4}{4}& \cnode{.4}{5} &  \rnode{6}{\Huge{\cdots}} \\
\psset{nodesep=3pt,arcangle=15,labelsep=2ex,linewidth=0.3mm,arrows=->,arrowsize=1mm
3} \nccurve[angleA=160,angleB=200,ncurv=4]{0}{0}
\uput[u](0.95,1.95){x_0}\uput[u](2.85,1.95){x_1}\uput[u](4.75,1.95){x_2}\uput[u](6.65,1.95){x_3}\uput[u](8.55,1.95){x_4}\uput[u](10.65,1.95){x_5}
\uput[u](-0.75,1.95){y_0}\uput[u](1.9,2.7){y_1}\uput[u](3.8,2.7){y_2}\uput[u](5.7,2.7){y_3}\uput[u](7.6,2.7){y_4}\uput[u](9.5,2.7){y_5}
\nccurve[angleA=70,angleB=110,ncurv=4]{1}{1}
\nccurve[angleA=70,angleB=110,ncurv=4]{2}{2}
\nccurve[angleA=70,angleB=110,ncurv=4]{3}{3}
\nccurve[angleA=70,angleB=110,ncurv=4]{4}{4}
\nccurve[angleA=70,angleB=110,ncurv=4]{5}{5}
 \ncarc{0}{1}\ncarc{1}{2}\ncarc{2}{3}
\ncarc{3}{4}\ncarc{4}{5} \ncarc{5}{6}
\psset{labelsep=-3.40ex}\nput{90}{0}{0}
\psset{labelsep=-3.40ex}\nput{90}{1}{1}
\psset{labelsep=-3.40ex}\nput{90}{2}{2}
\psset{labelsep=-3.40ex}\nput{90}{3}{3}
\psset{labelsep=-3.40ex}\nput{90}{4}{4}
\psset{labelsep=-3.40ex}\nput{90}{5}{5}
\end{psmatrix}
$$
\end{center}
\vspace{-1.2cm}
while $P_L$ is a pure death random walk on $\mathbb{Z}_{\geq0}$ with absorbing state at 0 with diagram
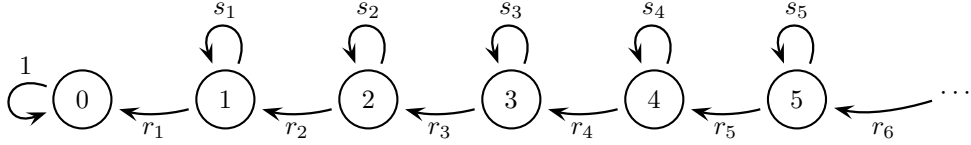
\vspace{0.7cm}
\begin{center}
$$\begin{psmatrix}[colsep=1.9cm]
  \cnode{.4}{0}& \cnode{.4}{1} & \cnode{.4}{2}& \cnode{.4}{3}& \cnode{.4}{4}& \cnode{.4}{5} &  \rnode{6}{\Huge{\cdots}} \\
\psset{nodesep=3pt,arcangle=15,labelsep=2ex,linewidth=0.3mm,arrows=->,arrowsize=1mm
3} \nccurve[angleA=160,angleB=200,ncurv=4]{0}{0}
\uput[u](0.95,1.1){r_1}\uput[u](2.85,1.1){r_2}\uput[u](4.75,1.1){r_3}\uput[u](6.65,1.1){r_4}\uput[u](8.55,1.1){r_5}\uput[u](10.65,1.1){r_6}
\uput[u](-0.75,1.95){1}\uput[u](1.9,2.7){s_1}\uput[u](3.8,2.7){s_2}\uput[u](5.7,2.7){s_3}\uput[u](7.6,2.7){s_4}\uput[u](9.5,2.7){s_5}
\nccurve[angleA=70,angleB=110,ncurv=4]{1}{1}
\nccurve[angleA=70,angleB=110,ncurv=4]{2}{2}
\nccurve[angleA=70,angleB=110,ncurv=4]{3}{3}
\nccurve[angleA=70,angleB=110,ncurv=4]{4}{4}
\nccurve[angleA=70,angleB=110,ncurv=4]{5}{5}
\ncarc{1}{0} \ncarc{2}{1}\ncarc{3}{2}
\ncarc{4}{3} \ncarc{5}{4} \ncarc{6}{5}
\psset{labelsep=-3.40ex}\nput{90}{0}{0}
\psset{labelsep=-3.40ex}\nput{90}{1}{1}
\psset{labelsep=-3.40ex}\nput{90}{2}{2}
\psset{labelsep=-3.40ex}\nput{90}{3}{3}
\psset{labelsep=-3.40ex}\nput{90}{4}{4}
\psset{labelsep=-3.40ex}\nput{90}{5}{5}
\end{psmatrix}
$$
\end{center}
\vspace{-1.2cm}
One could have performed the factorization the other way around like
\begin{equation}\label{Pxy2}
P=\begin{pmatrix}
b_0&a_0&\\
c_1&b_1&a_1&\\
&\ddots&\ddots&\ddots
\end{pmatrix}=\begin{pmatrix}
\tilde s_0&0&\\
\tilde r_1&\tilde s_1&0&\\
&\ddots&\ddots&\ddots
\end{pmatrix}\begin{pmatrix}
\tilde y_0&\tilde x_0&\\
0&\tilde y_1&\tilde x_1&\\
&\ddots&\ddots&\ddots
\end{pmatrix}=\widetilde P_L\widetilde P_U,
\end{equation}
in which case we have a LU factorization with relations
\begin{align*}
a_n&=\tilde s_{n}\tilde x_n,\quad n\geq0\\
b_n&=\tilde r_{n}\tilde x_{n-1}+\tilde s_n\tilde y_n,\quad n\geq0,\\
c_n&=\tilde r_n\tilde y_{n-1},\quad n\geq1.
\end{align*}
As we see, the important difference between both cases is that in the UL factorization case there will be a \emph{free parameter} $y_0$ while in the LU factorization case the decomposition will be unique. As before, the irreducibility conditions force us to take $\tilde x_n,\tilde s_{n}>0, n\geq0$ and $\tilde y_n,\tilde r_{n+1}>0, n\geq0$.


\begin{lemma}
Let $P=P_UP_L$ like in \eqref{Pxy}. Then $P_U$ is stochastic if and only if $P_L$ is stochastic. The same result holds for the LU decomposition \eqref{Pxy2}.
\end{lemma}
\begin{proof}
Assume that $P_U$ is stochastic. We will prove that $P_L$ is stochastic by induction. 
If $x_0+y_0=1$ then, using the first and second relation in \eqref{ULd}, we get
$$
s_1+r_1=\frac{a_0}{x_0}+\frac{b_0-y_0}{x_0}=\frac{1-y_0}{x_0}=1.
$$
Assume now that $s_n+r_n=1$. Then we have, using again \eqref{ULd}, that
\begin{align*}
s_{n+1}+r_{n+1}&=\frac{a_n}{x_n}+\frac{b_n-y_ns_n}{x_n}=\frac{1}{x_n}\left(a_n+b_n-y_n(1-r_n)\right)\\
&=\frac{1}{x_n}\left(a_n+b_n+c_n-y_n\right)=\frac{1}{x_n}\left(1-y_n\right)=1.
\end{align*}
Finally, the entries of $P_L$ are all nonnegative following \eqref{ULd}. On the other hand, if $P_L$ is stochastic, then all entries of $P_U$ are nonnegative and adding up the three relations in \eqref{ULd} we have
$$
x_n(s_{n+1}+r_{n+1})+y_n(s_n+r_n)=a_n+b_n+c_n=1.
$$
The proof for the LU decomposition is similar.
\end{proof}
From \eqref{ULd} it is clear that the we can generate all the sequences $x_n, y_n, s_n, r_n$ directly from $a_n, b_n, c_n$ and there will be an unique \emph{free parameter} $y_0$ (something that it will not be true for the LU decomposition). Indeed, we first calculate alternatively the sequences $s_{n+1},n\geq0,$ and $y_{n},n\geq1,$ by using the first and third relation in \eqref{ULd}. The sequences $x_n$ and $r_n$ satisfy $x_n=1-y_n, n\geq0,$ and $r_n=1-s_n,n\geq1$. The best order in which all sequences are calculated is ($y_0$ is a free parameter and $s_0=1$), $s_1,r_1,y_1,x_1,s_2,r_2,y_2,x_2,\ldots$. Another way to generate these sequences will be by using certain recurrence relations as the next lemma shows.

\begin{lemma}\label{lem22}
Let $x_n, y_n, s_n, r_n$ the sequences obtained by \eqref{Pxy} with $P_U$ having all its rows summing up to 1 (or $P_L$). Then the sequence $y_n$ satisfies the recurrence relation ($y_0$ is a free parameter)
\begin{equation}\label{recy}
y_{n+1}=\frac{c_{n+1}}{1-\D\frac{a_n}{1-y_n}},\quad n\geq0,\quad y_0\in\mathbb{R},
\end{equation}
while the sequence $s_n$ satisfies the recurrence relation
\begin{equation}\label{recs}
s_{n+1}=\frac{a_{n}}{1-\D\frac{c_n}{1-s_n}},\quad n\geq1,\quad s_1=\frac{a_0}{1-y_0}.
\end{equation}
\end{lemma}
\begin{proof}
From the first and third relation in \eqref{ULd} we have
\begin{align*}
y_{n+1}&=\frac{c_{n+1}}{r_{n+1}}=\frac{c_{n+1}}{1-s_{n+1}}=\frac{c_{n+1}}{1-\D\frac{a_n}{x_n}}=\frac{c_{n+1}}{1-\D\frac{a_n}{1-y_n}},\\
s_{n+1}&=\frac{a_{n}}{x_n}=\frac{a_{n}}{1-y_n}=\frac{a_{n}}{1-\D\frac{c_n}{1-s_n}}.
\end{align*}

\end{proof} 
\begin{remark}
A simpler way of writing \eqref{recy} and \eqref{recs} is calculating alternatively both sequences by using
\begin{equation}\label{recys}
y_{n+1}=\frac{c_{n+1}}{1-s_{n+1}},\quad s_{n+1}=\frac{a_n}{1-y_n},\quad n\geq0.
\end{equation}
\end{remark}
\begin{remark}
For the LU decomposition we have similar recurrence relations, but now for $\tilde r_n$ and $\tilde x_n$. Indeed,
\begin{equation}\label{recr}
\tilde r_{n+1}=\frac{c_{n+1}}{1-\D\frac{a_n}{1-\tilde r_n}}=\frac{c_{n+1}}{1-\tilde x_n},\quad n\geq0,\quad \tilde r_0=0,
\end{equation}
and
\begin{equation}\label{recx}
\tilde x_{n+1}=\frac{a_{n+1}}{1-\D\frac{c_{n+1}}{1-\tilde x_n}}=\frac{a_{n+1}}{1-\tilde r_{n+1}},\quad n\geq0,\quad \tilde x_0=a_0.
\end{equation}
Unlike the UL decomposition, there is no free parameter in this case.
\end{remark}
The previous lemma gives one way to obtain recursively all coefficients $x_n, y_n, s_n, r_n$ in terms of a free parameter $y_0$. Lemma \ref{lem22} does not say anything about the positivity of the coefficients $x_n, y_n, s_n, r_n$. This will be the goal of the next proposition, where we will study how to choose appropriately the parameter $y_0$ in such a way that $P_U$ and $P_L$ are both stochastic. Before that, we need to introduce some notation about \emph{continued fractions}. We recommend the reference \cite{Wa} for the reader unfamiliar with the subject.

Let $H$ be the continued fraction generated by alternatively choosing $a_n$ and $c_n$, i.e.
\begin{equation}\label{cfUL}
H=1-\cfrac{a_0}{1-\cfrac{c_1}{1-\cfrac{a_1}{1-\cfrac{c_2}{1-\cdots}}}}.
\end{equation}
Continued fractions admit different notations. In this paper we will also use
\begin{equation*}
\mbox{$H=\contFrac*{1 // a_0 // 1 // c_1 // 1 // a_1 // 1 // c_2 // \cdots}$}.
\end{equation*}
The sequence $a_0,c_1,a_1,c_2,\ldots$ is called the sequence of \emph{partial numerators} of the continued fraction. Consider the so-called convergents $(h_n)_{n\geq0}$ of the continued fraction $H$, given by the sequence of truncated continued fractions of $H$, i.e.
\begin{align*}
\mbox{$h_{2n}=\contFrac*{1 // a_0 // 1 // c_1// 1}-\cdots\contFrac*{// c_n // 1}$},\quad n\geq0,\\
\mbox{$h_{2n+1}=\contFrac*{1 // a_0 // 1 // c_1// 1}-\cdots\contFrac*{// a_n // 1}$},\quad n\geq0.
\end{align*}
It is well known (see \cite{Wa}) that the convergents of a continued fraction can be written in the form
\begin{equation}\label{hhh}
h_n=\frac{A_n}{B_n},
\end{equation}
where the numbers $A_n, B_n,$ can be calculated recursively by the formulas
\begin{align*}
A_n&=A_{n-1}-\xi_nA_{n-2},\quad A_{-1}=1,\quad A_0=1,\quad n\geq1,\\
B_n&=B_{n-1}-\xi_nB_{n-2},\quad B_{-1}=0,\quad B_0=1,\quad n\geq1.
\end{align*}
Here $\xi_n$ is the sequence given by $\xi_{2n}=c_n$ and $\xi_{2n+1}=a_n$ for $n\geq1$. In terms of $h_{2n}$ and $h_{2n+1}$ we have
\begin{align*}
A_{2n}&=A_{2n-1}-c_nA_{2n-2},\quad n\geq1,\\
B_{2n}&=B_{2n-1}-c_nB_{2n-2},\quad n\geq1.
\end{align*}
and
\begin{align*}
A_{2n+1}&=A_{2n}-a_nA_{2n-1},\quad n\geq0,\\
B_{2n+1}&=B_{2n}-a_nB_{2n-1},\quad n\geq0.
\end{align*}
Using the relations above it is not hard to prove that
\begin{align*}
A_{2n+2}B_{2n+1}-A_{2n+1}B_{2n+2}&=-a_0c_1a_1c_2\cdots a_{n-1}c_n,\quad n\geq0,\\
A_{2n+1}B_{2n}-A_{2n}B_{2n+1}&=-a_0c_1a_1c_2\cdots c_na_n,\quad n\geq0,
\end{align*}
or, in other words,
\begin{align}
\label{hhn1}h_{2n+2}-h_{2n+1}&=\frac{A_{2n+2}}{B_{2n+2}}-\frac{A_{2n+1}}{B_{2n+1}}=-\frac{a_0c_1a_1c_2\cdots a_{n-1}c_n}{B_{2n+1}B_{2n+2}},\quad n\geq0,\\
\label{hhn2}h_{2n+1}-h_{2n}&=\frac{A_{2n+1}}{B_{2n+1}}-\frac{A_{2n}}{B_{2n}}=-\frac{a_0c_1a_1c_2\cdots c_na_n}{B_{2n}B_{2n+1}},\quad n\geq0.
\end{align}
From here we get the following:

\begin{proposition}\label{lemcf}
Let $H$ the continued fraction given by \eqref{cfUL} and the corresponding convergents $h_n=A_n/B_n$ defined by \eqref{hhh}. Assume that 
\begin{equation}\label{AnBn}
0<A_n<B_n,\quad n\geq1.
\end{equation}
Then $H$ is convergent. Moreover, let $P=P_UP_L$ like in \eqref{Pxy}. Then, both $P_U$ and $P_L$ are stochastic matrices if and only if we choose $y_0$ in the following range
\begin{equation}\label{yy0r}
0\leq y_0\leq H.
\end{equation}
\end{proposition}
\begin{proof}
First, following the assumptions \eqref{AnBn} and properties \eqref{hhn1} and \eqref{hhn2}, we have that (recall that $0<a_n,c_{n+1}<1, n\geq0$)
$$
0<\cdots<h_{2n+2}<h_{2n+1}<h_{2n}<\cdots<h_1<h_0=1.
$$
Therefore $h_n$ is a bounded strictly decreasing sequence, so it is convergent to $H$. 

Assume first that both $P_U$ and $P_L$ are stochastic. That means that $0\leq y_0<1$, $0<x_n,s_{n+1}<1, n\geq0$ and $0<y_n,r_{n}<1, n\geq1$, following the irreducibility conditions for the matrix $P$. Therefore, we already have the lower bound for $y_0$, i.e. $0\leq y_0$. In order to proof that $y_0\leq H$ we will use the definition of $y_n$ and $s_n$ in \eqref{recys}. Consider the sequence of numbers $\alpha_n,n\geq0,$ given by $\alpha_{2n}=y_n,\alpha_{2n+1}=s_{n+1},n\geq0$. We will see that the condition $\alpha_n<1$ implies that $y_0<h_n$ for every $n$. Therefore, since $h_n$ is a positive bounded and decreasing sequence, we will have \eqref{yy0r}. 

For $\alpha_0=y_0$ it is clear that $0\leq y_0<1=h_0$. For $\alpha_1=s_1$, we have, by definition,
$$
\alpha_1=s_1=\frac{a_0}{1-y_0}<1\Leftrightarrow y_0<1-a_0=h_1.
$$
For $\alpha_2=y_1$ we have, using \eqref{recys} and the previous bound, that
$$
\alpha_2=y_1=\frac{c_1}{1-s_1}<1\Leftrightarrow s_1<1-c_1\Leftrightarrow \frac{a_0}{1-y_0}<1-c_1\Leftrightarrow y_0<1-\frac{a_0}{1-c_1}=h_2.
$$
For an even index $2n$ we have, using \eqref{recys} and all the previous bounds, that
\begin{align*}
\alpha_{2n}&=y_{n}=\frac{c_n}{1-s_n}<1\Leftrightarrow s_n<1-c_n\Leftrightarrow \frac{a_{n-1}}{1-y_{n-1}}<1-c_n\Leftrightarrow y_{n-1}<1-\frac{a_{n-1}}{1-c_n}\\
&\Leftrightarrow s_{n-1}<1-\frac{c_{n-1}}{1-\D\frac{a_{n-1}}{1-c_n}}\Leftrightarrow \cdots \Leftrightarrow y_0<\mbox{$\contFrac*{1 // a_0 // 1 // c_1// 1}-\cdots\contFrac*{// c_n // 1}=h_{2n}$}.
\end{align*}
Similarly, for an odd index $2n+1$ we have, using \eqref{recys} and all the previous bounds, that
\begin{align*}
\alpha_{2n+1}&=s_{n+1}=\frac{a_n}{1-y_n}<1\Leftrightarrow y_n<1-a_n\Leftrightarrow \frac{c_{n}}{1-s_n}<1-a_n\Leftrightarrow s_{n}<1-\frac{c_{n}}{1-a_n}\\
&\Leftrightarrow y_{n-1}<1-\frac{a_{n-1}}{1-\D\frac{c_{n}}{1-a_n}}\Leftrightarrow \cdots \Leftrightarrow y_0<\mbox{$\contFrac*{1 // a_0 // 1 // c_1// 1}-\cdots\contFrac*{// a_n // 1}=h_{2n+1}$}.
\end{align*}
On the contrary, if \eqref{yy0r} holds, in particular we have that $0\leq y_0\leq H<h_n$ for every $n\geq0$. Following the same steps as before, using an argument of strong induction will lead us to the fact that both $P_U$ and $P_L$ are stochastic matrices with the conditions that $0\leq y_0<1$, $0<x_n,s_{n+1}<1, n\geq0$ and $0<y_n,r_{n}<1, n\geq1$. 
\end{proof}

\begin{remark}\label{remLU}
For the LU decomposition there is no free parameter, so the positivity condition comes in terms of an upper bound of the coefficient $\tilde y_0=a_0$. Indeed, one must have 
$$0\leq a_0\leq \widetilde H,$$ where
\begin{equation}\label{cfLU}
\mbox{$\widetilde H=\contFrac*{1 // c_1 // 1 // a_1 // 1 // c_2 // 1 // a_2 // \cdots}$},
\end{equation}
as long as we have $0<\widetilde A_n<\widetilde B_n,n\geq1,$ where $\tilde h_n=\widetilde A_n/\widetilde B_n$ are the convergents of $\widetilde H$. The proof is similar to the one in the previous proposition, using now the recurrence relations \eqref{recr} and \eqref{recx}.
\end{remark}

\begin{remark}\label{remCS}
The subject of calculating the exact value of a continued fraction or establishing its convergence is a delicate matter. There are many results allowing one to decide whether a continued fraction is convergent or not. We refer again to \cite{Wa} to find a collection of convergence results for different types of continued fractions. For instance, one of the oldest results is the \emph{Worpitzky's Theorem}, which states that a continued fraction converges if all partial numerators have moduli less than $1/4$. Another interesting case is given if the sequence of partial numerators is a \emph{chain sequence} (see for instance Theorem 3.1 of \cite{Ch}). Indeed, consider the continued fraction 
\begin{equation*}
\mbox{$C=\contFrac*{1 // \alpha_1 // 1 // \alpha_2 // 1 // \alpha_3 // 1 // \alpha_4 // \cdots}$}.
\end{equation*}
Assume that $\alpha_n$ can be written as $\alpha_n=(1-m_{n-1})m_n$, where $0\leq m_0<1$ and $0<m_n<1$ for $n\geq1$. Then
$$
C=m_0+\frac{1-m_0}{1+L},
$$
where
\begin{equation}\label{LL}
L=\sum_{n=1}^{\infty}\frac{m_1m_2\cdots m_n}{(1-m_1)(1-m_2)\cdots(1-m_n)}.
\end{equation}
Moreover, in that case the convergents satisfy the condition \eqref{AnBn} of the Proposition \ref{lemcf}. Indeed, following Section 3.3 of \cite{Ch}, it is possible to see that the convergents $C_n=A_n/B_n$ of $C$ satisfy
\begin{align*}
A_n&=\prod_{k=1}^n(1-m_k)+m_0\left(\sum_{k=1}^{n-1}m_1\cdots m_k(1-m_{k+1})\cdots (1-m_n)+\prod_{k=1}^nm_k\right)\\
B_n&=\prod_{k=1}^n(1-m_k)+\sum_{k=1}^{n-1}m_1\cdots m_k(1-m_{k+1})\cdots (1-m_n)+\prod_{k=1}^nm_k.
\end{align*}
From here it is clear that $0<A_n<B_n, n\geq 1$. The argument above will be used for the Jacobi polynomials in Section \ref{sec5}.
\end{remark}

\begin{remark}\label{remy0}
Observe that the case $y_0=0$ allows for UL decomposition with a \emph{free parameter $s_0$} not necessarily equal to 1. Indeed, $x_0=1$, the sequences $y_n, x_n,n\geq1,$ can be obtained by \eqref{recy} and they are independent of $s_0$. The same is true for the sequences $r_n, s_n,n\geq1,$ which can be obtained by \eqref{recs} and they are independent of $s_0$. Therefore we can take any $s_0$ such that $0\leq s_0\leq1$. The probabilistic meaning of taking $y_0=0$ is that the random walk $P_U$ has a repelling barrier at state $0$. For the random walk $P_L$ it is convenient to think of an ignored state (say $-1$) as an absorbing state from the state $0$ with probability $1-s_0$.
\end{remark}

\section{Stochastic Darboux transformations}\label{sec3}

Now that we know the conditions under which a stochastic tridiagonal matrix can be decomposed as a UL (or LU) factorization where both factors are still stochastic matrices, we can perform what is called a \emph{discrete Darboux transformation}. The Darboux transformation has a long history but probably the first reference of a discrete Darboux transformation like we study here appeared in \cite{MS} in connection with the Toda lattice. We explain now what is a Darboux transformation in our context.

If $P=P_UP_L$ as in \eqref{Pxy}, then by inverting the order of multiplication we obtain another tridiagonal matrix of the form
\begin{equation}\label{DarbT}
\widetilde{P}=P_LP_U=\begin{pmatrix}
s_0&0&\\
r_1&s_1&0&\\
&\ddots&\ddots&\ddots
\end{pmatrix}\begin{pmatrix}
y_0&x_0&\\
0&y_1&x_1&\\
&\ddots&\ddots&\ddots
\end{pmatrix}=\begin{pmatrix}
\tilde{b}_0&\tilde{a}_0&\\
\tilde{c}_1&\tilde{b}_1&\tilde{a}_1&\\
&\ddots&\ddots&\ddots
\end{pmatrix}.
\end{equation}
Now the new coefficients are given by
\begin{align*}
\nonumber\tilde{a}_n&=s_nx_{n},\quad n\geq0,\\
\tilde{b}_n&=r_nx_{n-1}+s_ny_n,\quad n\geq0,\\
\nonumber\tilde{c}_n&=r_ny_{n-1},\quad n\geq1.
\end{align*}
The matrix $\widetilde{P}$ is actually stochastic, since the multiplication of two stochastic matrices is again a stochastic matrix. Therefore it gives a \emph{family} of new random walks with coefficients $(\tilde{a}_n)_n$, $(\tilde{b}_n)_n$ and $(\tilde{c}_n)_n$ and depending on a free parameter $y_0$. In terms of a model driven by urn experiments (as we will see in the example below) the factorization $P=P_UP_L$ may be thought as two urn experiments, Experiment 1 and Experiment 2, respectively. We first perform the Experiment 1 and with the result we immediately perform the Experiment 2. The urn model for $\widetilde{P}=P_LP_U$ will proceed in the reversed order, first the Experiment 2 and with the result the Experiment 1. We will see more on this later for some specific examples.

The same can be done for the LU decomposition \eqref{Pxy2} of the form $P=\widetilde P_L\widetilde P_U$ but now we do not get a family of random walks and the transformation is unique. The corresponding Darboux transformation is
\begin{equation}\label{DarbT2}
\widehat{P}=\widetilde P_U\widetilde P_L=\begin{pmatrix}
\tilde y_0&\tilde x_0&\\
0&\tilde y_1&\tilde x_1&\\
&\ddots&\ddots&\ddots
\end{pmatrix}\begin{pmatrix}
\tilde s_0&0&\\
\tilde r_1&\tilde s_1&0&\\
&\ddots&\ddots&\ddots
\end{pmatrix}=\begin{pmatrix}
\hat{b}_0&\hat{a}_0&\\
\hat{c}_1&\hat{b}_1&\hat{a}_1&\\
&\ddots&\ddots&\ddots
\end{pmatrix}.
\end{equation}
The new coefficients are given by
\begin{align*}
\hat{a}_n&=\tilde x_{n}\tilde s_{n+1},\quad n\geq0,\\
\hat{b}_n&=\tilde x_{n}\tilde r_{n+1}+\tilde y_n\tilde s_n,\quad n\geq0,\\
\hat{c}_n&=\tilde y_{n}\tilde r_n,\quad n\geq1.
\end{align*}

\medskip

One important property of the Darboux transformation is how to transform the spectral measure associated with a random walk with one-step transition probability $P$. It is very well known (see \cite{KMc6}) that for every tridiagonal stochastic matrix $P$ (also known as a Jacobi matrix) there exists an unique positive measure $\omega$ defined on the interval $-1\leq x\leq1$. This is just a consequence of the very well known Spectral Theorem (or Favard's Theorem in the context of orthogonal polynomials). The Darboux transformation \eqref{DarbT} gives a family of random walks $\widetilde P$ which is also a Jacobi matrix. Therefore, there will exist an unique family of positive measures $\widetilde\omega$ associated with $\widetilde P$. It is possible to derive (see for instance Theorem 2 of \cite{GHH}) that the moments of this new $\widetilde\omega$ are given by 
$$
\widetilde\mu_0=1,\quad \widetilde\mu_n=y_0\mu_{n-1}, \quad n\geq1,
$$
where $\mu_n$ are the moments of $\omega$ and $y_0$ is the free parameter from the UL factorization. If the moment $\mu_{-1}=\int_{-1}^1d\omega(x)/x$ is well defined, then a candidate for the family of spectral measures is then
\begin{equation}\label{spme}
\widetilde\omega(x)=y_0\frac{\omega(x)}{x}+M\delta_0(x),\quad M=1-y_0\mu_{-1},
\end{equation}
where $\delta_0(x)$ is the Dirac delta located at $x=0$. This transformation of the spectral measure $\omega$ is also known as a \emph{Geronimus transformation}.

Similarly, for the LU decomposition, the corresponding Darboux transformation \eqref{DarbT2} $\widehat{P}$ gives rise to a Jacobi matrix and a spectral measure $\widehat\omega$. In this case, it is possible to see that this new spectral measure is given by
\begin{equation}\label{spme2}
\widehat\omega(x)=x\omega(x),
\end{equation}
or, in other words, a \emph{Christoffel transformation} of $\omega$. For more information about the connection between Darboux transformations and Geronimus or Christoffel transformations see \cite{GH,GHH, SZ, Yo, Ze}.

\begin{remark}\label{remy02}
Following Remark \ref{remy0} we can see now how the choice of $y_0=0$ affects the random walk $\widetilde P$ \eqref{DarbT} obtained from the Darboux transformation. Indeed, in this case we must have that $\tilde b_0=0$, $\tilde a_0=s_0$ \emph{and} $\tilde c_1=0$, so $\widetilde P$ will be a random walk with a free parameter $s_0$. The random walk $\widetilde P$ can be decomposed as
$$
\widetilde{P}=\left(
\begin{array}{c|c}
0 & \begin{array}{ccc} s_0&0&\cdots\end{array} \\
\hline
\begin{array}{c} 0\\0\\\vdots\end{array} & \breve P
\end{array}
\right),
$$
where $\breve P$ is stochastic, irreducible and does not depend on $s_0$. If $s_0<1$ there is a probability that the process is absorbed at state $-1$ if the random walk is at state $0$, given by $1-s_0$. But the random walk $\widetilde P$ can never reach the state $0$, unless it starts there. In terms of the spectral measure for $\widetilde P$ in \eqref{spme}  we see that if $y_0=0$ then $\widetilde\omega(x)=\delta_0(x)$ and we have a \emph{degenerate measure}.
\end{remark}

\section{Random walk with constant transition probabilities}\label{sec4}

In this section we apply the machinery developed in previous sections to the study of the random walk with one-step transition probability matrix given by
\begin{equation}\label{PPc}
P=\begin{pmatrix}
b_0&a_0&0&0\\
c&b&a&0&\\
0&c&b&a&\\
&&\ddots&\ddots&\ddots
\end{pmatrix},
\end{equation}
where $b_0+a_0=1$ and $a+b+c=1$. In order to apply Proposition \ref{lemcf} the continued fraction $H$ in \eqref{cfUL} can be written as
$$
H=1-\frac{a_0}{F},
$$
where
\begin{equation}\label{cfF}
\mbox{$F=\contFrac*{1 // c // 1 // a // 1 // c // 1 // a // \cdots}$}.
\end{equation}
We observe that $F$ is a periodic continued fraction of period 2. This means that $F$, if it converges, will be the solution of certain quadratic equation.
\begin{proposition}\label{propF}
Let $c\leq(1-\sqrt{a})^2$. Then the continued fraction $F$ \eqref{cfF} converges to
\begin{equation}\label{FFF}
F=\frac{1}{2}\left(1+a-c+\sqrt{(1+a-c)^2-4a}\right).
\end{equation}
\end{proposition}
\begin{proof}
Heuristically the expression of $F$ in \eqref{cfF} can be written as
$$
F=1-\frac{c}{1-\D\frac{a}{F}}.
$$
Then $F$ satisfies the quadratic equation
$$
F^2-(1+a-c)F+a=0,
$$
and the solutions are given by
$$
F_{\pm}=\frac{1}{2}\left(1+a-c\pm\sqrt{(1+a-c)^2-4a}\right).
$$
We now see that the condition $c\leq(1-\sqrt{a})^2$ implies that $F$ is actually a real number and the value of the continued fraction $F$ is exactly the solution having the larger modulus. A more rigorous proof can be done following Theorem 8.1 of \cite{Wa}.
\end{proof}

\begin{remark}
From here we already see that there are many cases of random walks with constant transition probabilities where \emph{it is not possible} to obtain a stochastic UL factorization. For instance, the case of a symmetric random walk with $c=1-a$ (and therefore $b=0$). In this case the condition $c\leq(1-\sqrt{a})^2$ is equivalent to $a\geq1$ which is possible only if $a=1$. But then $P$ would not be irreducible.
\end{remark}

\begin{corollary}\label{cororwc}
Let $P$ be like in \eqref{PPc} with $a_0\leq F$, where $F$ is given by \eqref{FFF}. Then $P=P_UP_L$ like in \eqref{Pxy} where both $P_U$ and $P_L$ are stochastic matrices if and only if we choose $y_0$ in the following range
$$
0\leq y_0\leq1-\frac{a_0}{F}.
$$
\end{corollary}
\begin{proof}
The condition $a_0\leq F$ gives that $1-a_0/F$ is actually a number between 0 and 1 since $F$ is always positive.
\end{proof}

There are two interesting cases:
\begin{itemize}
\item If $a=c$, then the continued fraction $F$ in \eqref{cfF} has period 1 and the condition $c\leq(1-\sqrt{a})^2$ is equivalent to $a\leq 1/4$. Then $F=\left(1+\sqrt{1-4a}\right)/2$ and 
$$
0\leq y_0\leq1-\frac{2a_0}{1+\sqrt{1-4a}},
$$
as long as $a_0\leq \left(1+\sqrt{1-4a}\right)/2$. Additionally we have $b=1-2a$.

\item  If $c=(1-\sqrt{a})^2$, then the continued fraction $F$ in \eqref{cfF} gives that $F=\sqrt{a}$ and 
$$
0\leq y_0\leq1-\frac{a_0}{\sqrt{a}},
$$
as long as $a_0\leq\sqrt{a}$. Additionally we have $b=2(\sqrt{a}-a)$.
\end{itemize}

It is possible to calculate the coefficients $x_n,y_n,s_n,r_n$ of the UL decomposition using Lemma \ref{lem22}. These coefficients will depend on $a_0$ and $y_0$. In general the formulas are quite cumbersome, but in the special cases introduced above, these formulas simplify considerably. For illustration purposes we will give the coefficients $x_n,y_n,s_n,r_n$ for two cases:

\medskip

a) $a=c=1/4$ (and therefore $b=1/2$). We have $F=1/2$ and then $0\leq y_0\leq 1-2a_0$ as long as $a_0\leq1/2$. A simple computation shows
\begin{align*}
y_n&=\frac{2a_0+(2n-1)(1-y_0-2a_0)}{4a_0+4n(1-y_0-2a_0)},\quad x_n=1-y_n,\quad n\geq1,\\
s_n&=\frac{a_0+(n-1)(1-y_0-2a_0)}{2a_0+(2n-1)(1-y_0-2a_0)},\quad r_n=1-s_n,\quad n\geq 1,\quad s_0=1.
\end{align*}
We clearly see from these formulas that $0<x_n,y_n<1, n\geq0$ and $0<s_n,r_n<1, n\geq1$. There are certain choices of $y_0$ where the coefficients are independent of $a_0$. For instance, taking $y_0=1-ka_0, k\geq2$, with $a_0\leq1/k$, gives
\begin{align}
\label{yxsn2}y_n&=\frac{2+(2n-1)(k-2)}{4+4n(k-2)},\quad x_n=1-y_n,\quad n\geq1,\\
\nonumber s_n&=\frac{1+(n-1)(k-2)}{2+(2n-1)(k-2)},\quad r_n=1-s_n,\quad n\geq 1,\quad s_0=1.
\end{align}
For $k=2$ we see that $x_n=y_n=s_n=r_n=1/2, n\geq1$. We will give a probabilistic implementation in terms of urns models below.

\begin{remark}
Following Remark \ref{remLU}, for the LU factorization we must have $0\leq a_0\leq\widetilde H=F$, where $F$ is the continued fraction \eqref{cfF}. As long as we are in the conditions of Proposition \ref{propF} we will have a convergent continued fraction $F$ which value is given by \eqref{FFF}. The coefficients $\tilde x_n,\tilde y_n,\tilde s_n,\tilde r_n$ in \eqref{Pxy2} can be calculated in the same way as before. For instance, for the case of $a=c=1/4$ (and therefore $b=1/2$), we have $F=1/2$ and then $a_0$ must be taken in the range $0\leq a_0\leq1/2$. Then a simple computation shows that
\begin{align*}
\tilde x_n&=\frac{a_0+n(1-2a_0)}{2a_0+2n(1-2a_0)},\quad \tilde y_n=1-\tilde x_n,\quad n\geq0,\\
\tilde r_n&=\frac{2a_0+(2n-1)(1-2a_0)}{4a_0+4n(1-2a_0)},\quad \tilde s_n=1-\tilde r_n,\quad n\geq 1,\quad \tilde s_0=1.
\end{align*}
\end{remark}

\medskip

b) $a=1/9$, $c=(1-\sqrt{a})^2=4/9$ (and therefore $b=4/9$). We have $F=1/3$ and then $0\leq y_0\leq 1-3a_0$ as long as $a_0\leq1/3$. In this case we have to differentiate between odd and even cases. Therefore
\begin{align*}
y_{2n}&=\frac{6a_0+(6n-1)(1-y_0-3a_0)}{9a_0+9n(1-y_0-3a_0)},\quad y_{2n-1}=\frac{12a_0+(12n-8)(1-y_0-3a_0)}{18a_0+9(2n-1)(1-y_0-3a_0)},\quad n\geq1,\\
s_0&=1, s_1=\frac{a_0}{1-y_0},s_{n}=\frac{1}{9(1-y_{n-1})},n\geq2,\quad x_n=1-y_n,\quad r_n=1-s_n,\quad n\geq1.
\end{align*}
As before, if we take $y_0=1-ka_0, k\geq3,$ with $a_0\leq1/k$, we get the coefficients independent of $a_0$.
\medskip

For any values of $a$ and $c$ the sequence $y_n$ seems to have the following form ($a_0\leq F$) 
$$
y_n=\frac{\beta_1a_0+\xi_n^1(1-y_0-a_0/F)}{\beta_2a_0+\xi_n^2(1-y_0-a_0/F)},\quad n\geq1,
$$
for certain constants $\beta_i,i=1,2,$ and sequences $\xi_n^i, i=1,2,$ with $\beta_1<\beta_2$ and $\xi_n^1<\xi_n^2, n\geq 1$.

\subsection{The spectral measure}
It is well known how to get the spectral measure associated with the random walk \eqref{PPc} using the techniques developed in \cite{KMc6} (see also \cite{CG}). First we notice that $P$ can be decomposed as
$$
P=\left(
\begin{array}{c|c}
b_0 & \begin{array}{ccc} a_0&0&\cdots\end{array} \\
\hline
\begin{array}{c} c\\0\\\vdots\end{array} & \breve P
\end{array}
\right),\quad \breve P=\begin{pmatrix}
b&a&&\\
c&b&a&&\\
&\ddots&\ddots&\ddots
\end{pmatrix}.
$$
The spectral measure associated with $\breve P$ can be easily calculated since the associated process (or the process built from $\breve P$ by deleting the first row and column of $\breve P$) is exactly the same as $\breve P$. Therefore the Stieltjes transform of the spectral measure $\psi$ is given by ($b_0=1-a_0$ and $b=1-a-c$)
$$
B(z;\psi)\doteq\int_{-1}^1\frac{d\psi(x)}{x-z}=\frac{b-z\pm\sqrt{(z-\sigma_+)(z-\sigma_-)}}{2ac},\quad \sigma_{\pm}=1-(\sqrt{a}\mp\sqrt{c})^2.
$$
An application of the Stieltjes-Perron inversion formula gives that the spectral measure has only a continuous part given by
$$
\psi(x)=\frac{1}{2\pi ac}\sqrt{(x-\sigma_-)(\sigma_+-x)},\quad x\in[\sigma_-,\sigma_+]\subseteq[-1,1].
$$
The orthogonal polynomials associated with this measure are basically affine transformations of the Chebyshev polynomials of the second kind.

Now we compute the spectral measure $\omega$ of the process $P$. For that we use the same trick as before but now we know the expression of $B(z;\psi)$. After straightforward computations we have that the Stieltjes transform of the spectral measure $\omega$ is given by ($b_0=1-a_0$ and $b=1-a-c$)
$$
B(z;\omega)=\frac{2a-aa_0-a_0+a_0c+(a_0-2a)z+a_0\sqrt{(z-\sigma_+)(z-\sigma_-)}}{2(1-z)((a_0-a)z+a-aa_0+a_0c+a_0^2-a_0)}.
$$
Observe that the Stieltjes transform has two real poles (or one if $a_0=a$). One at $z=1$ and the other at $z=\gamma$, where
$$
\gamma=\frac{a_0-a+aa_0-a_0c-a_0^2}{a_0-a}.
$$
This means that the measure $\omega$ will consist of a continuous density plus possibly some delta masses located at $x=1$ and/or $x=\gamma$ with certain weights, i.e.
\begin{equation}\label{meas}
\omega(x)=\omega_c(x)+\omega\left(\{1\}\right)\delta_1(x)+\omega\left(\{\gamma\}\right)\delta_\gamma(x).
\end{equation}
Again, using the Stieltjes-Perron inversion formula, the continuous part of the measure is given by
\begin{equation}\label{measc}
\omega_c(x)=\frac{a_0\sqrt{(x-\sigma_-)(\sigma_+-x)}}{2\pi(1-x)((a_0-a)x+a-aa_0+a_0c+a_0^2-a_0)},\quad x\in[\sigma_-,\sigma_+]\subseteq[-1,1].
\end{equation}
The discrete masses in $\omega$ \eqref{meas} come from the residues at the simple poles of $B(z;\omega)$. A simple computation shows that
$$
\omega\left(\{1\}\right)=\frac{c-a}{a_0+c-a}\chi_{\{c>a\}},\quad \omega\left(\{\gamma\}\right)=\frac{(a_0-a)^2-ac}{(a_0-a)^2-ac+a_0c}\chi_{\{(a_0-a)^2>ac\}},
$$
where $\chi_A$ is the indicator function. It is also possible to see that the location of $\gamma$ is outside of $[\sigma_-,\sigma_+]$ and that with the condition $(a_0-a)^2>ac$ we must have $\gamma\in[-1,1]\setminus[\sigma_-,\sigma_+]$. The condition $(a_0-a)^2>ac$ can also be written in terms of $a_0$ if we choose either $a_0<\max\{0,a-\sqrt{ac}\}$ or $a_0>\min\{a+\sqrt{ac},1\}$ whenever one or both relations hold. In other words
\begin{equation}\label{conds}
\left[\sqrt{a}>\sqrt{c}\;\;\;\mbox{AND}\;\;\; a_0<a-\sqrt{ac}\right]\;\;\;\mbox{OR}\;\;\; \left[\sqrt{a}+\sqrt{c}<1/\sqrt{a}\;\;\;\mbox{AND}\;\;\; a_0>a+\sqrt{ac}\right].
\end{equation}

Let us recall the two interesting cases introduced after Corollary \ref{cororwc}:
\begin{itemize}
\item If $a=c$ we have that $\sigma_-=1-4a, \sigma_+=1$, $\omega\left(\{1\}\right)=0$ (no mass at $x=1$) and (the left hand side of \eqref{conds} does not count)
$$
\omega\left(\{\gamma\}\right)=\frac{a_0-2a}{a_0-a}\chi_{\{a<1/2, a_0>2a\}},\quad\gamma=1-\frac{a_0^2}{a_0-a}.
$$
Therefore the spectral measure \eqref{meas} is given by
\begin{equation}\label{spmeac}
\omega(x)=\omega_c(x)+\frac{a_0-2a}{a_0-a}\chi_{\{a<1/2, a_0>2a\}}\delta_\gamma(x).
\end{equation}
where
$$
\omega_c(x)=\frac{a_0}{2\pi[(a_0-a)x+a-a_0+a_0^2]}\sqrt{\frac{x-1+4a}{1-x}},\quad x\in[1-4a,1].
$$

\item If $c=(1-\sqrt{a})^2$ we have that $\sigma_-=0, \sigma_+=4(\sqrt{a}-a)$, 
$$
\omega\left(\{1\}\right)=\frac{1-2\sqrt{a}}{1+a_0-2\sqrt{a}}\chi_{\{a<1/4\}},
$$
and
$$
\omega\left(\{\gamma\}\right)=\left[\frac{a}{a-a_0}+\frac{a_0}{a_0+1-2\sqrt{a}}\right]\chi_{\left\{\{a<1/4, a_0<2a-\sqrt{a}\}\cup\{a_0>\sqrt{a}\}\right\}},\quad\gamma=\frac{(a_0-\sqrt{a})^2}{a-a_0}.
$$
Therefore the spectral measure $\omega$ is given by \eqref{meas} where the continuous part $\omega_c$ is
$$
\omega_c(x)=\frac{a_0\sqrt{x(4\sqrt{a}-4a-x)}}{2\pi(1-x)((a_0-a)x+a_0^2+a-2a_0\sqrt{a})},\quad x\in[0,4(\sqrt{a}-a)].
$$ 

\end{itemize}

\medskip 

As we mentioned in Section \ref{sec3} the spectral measure associated with the Darboux transformation is given by the Geronimus transformation \eqref{spme} of the measure $\omega$ in \eqref{meas}. For this to work we need to have the moment $\mu_{-1}=\int\omega(x)/xdx$ well defined. The delta at $x=1$ does not add any additional problem to this moment. The delta at $x=\gamma$ may produce a problem if $\gamma=0$, which is possible if
$$
c=\frac{(1-a_0)(a_0-a)}{a_0},\quad ac<(a_0-a)^2.
$$
Substituting the value of $c$ in the inequality gives that $a_0>\sqrt{a}$. With this assumption we have that $F=a_0$. According to Corollary \ref{cororwc} the only possible choice is for $y_0=0$, which we already know gives a degenerate measure. If $ac\geq(a_0-a)^2$ then there is no delta at $x=\gamma$ and since $\gamma$ is always located outside of $[\sigma_-,\sigma_+]$ there is no problem with the continuous part $\omega_c$ of $\omega$. 

As for the continuous part $\omega_c$ \eqref{measc}, since we are dividing by $x$ in the Geronimus transformation, there may be a problem of integrability if $0\in(\sigma_-,\sigma_+)$. On one side, it is easy to see that $\sigma_+>0$. On the other side, we have that $\sigma_-<0$ if and only if $\sqrt{a}+\sqrt{c}>1$. In other words, $ac>b^2/4$. But in this case we should have that $c>(1-\sqrt{a})^2$ and this is a contradiction with the assumptions of Proposition \ref{propF}.

As a conclusion, every time we want to perform a Darboux transformation with stochastic factors, it is always possible to find the associated spectral measure via the Geronimus transformation of $\omega$ and there will be no integrability problems. For the LU factorization the associated spectral measure is the Christoffel transformation \eqref{spme2} as long as $0\leq a_0\leq F$. But this does not bring in any integrability problems to $\omega$.

\subsection{A factorization and its associated urn models}

We will give a probabilistic implementation of the UL factorization in terms of a family of urn models for the case of $a=c=1/4$ and $y_0=1-ka_0, k\geq2$, with $a_0\leq1/k$ given by \eqref{yxsn2} (the rest of cases such as the LU factorization or the case $a=1/9,c=4/9$ can be treated in a similar way with small modifications). We will assume that $k\geq2$ is a positive integer (certain number of balls). The stochastic matrix $P$ is then given by
\begin{equation*}
P=\begin{pmatrix}
1-a_0&a_0&0&0\\
1/4&1/2&1/4&0&\\
0&1/4&1/2&1/4&\\
&&\ddots&\ddots&\ddots
\end{pmatrix}.
\end{equation*}
This is already a very simple random walk and one could argue that analyzing it any further is unnecessary. It is given here as an illustration of what will be done later in a more complicated case.

From \eqref{yxsn2} we see that the factors $P_U$ and $P_L$ in \eqref{Pxy} are given by
\begin{equation*}
P_U=\begin{pmatrix}
1-ka_0&ka_0&0&0\\
0&\frac{k}{4k-4}&\frac{3k-4}{4k-4}&0&\\
0&0&\frac{3k-4}{8k-12}&\frac{5k-8}{8k-12}&\\
&&\ddots&\ddots&\ddots
\end{pmatrix},\;P_L=\begin{pmatrix}
1&0&0&0\\
\frac{k-1}{k}&\frac{1}{k}&0&0&\\
0&\frac{2k-3}{3k-4}&\frac{k-1}{3k-4}&0&\\
&&\ddots&\ddots&\ddots
\end{pmatrix}.
\end{equation*}
Each one of these matrices $P_U$ and $P_L$ will represent an experiment in terms of an urn model, which we call Experiment 1 and Experiment 2, respectively. Observe that $k\geq2$ is a free positive integer, so we will have a \emph{family} of urn experiments which combined gives the \emph{same} urn model for $P$. There is a very simple case where $k=2$ and then $x_n=y_n=s_n=r_n=1/2, n\geq1$, so we will assume that $k\geq3$.

The Experiment 1 (for $P_U$) consists of a discrete time pure birth random walk on the nonnegative integers $\mathbb{Z}_{\geq0}$ (see diagram below \eqref{ULd}) where each state represents the number of blue balls in the urn. Assume that this urn sits in a bath consisting of an infinite number of blue and red balls. If the state of the system is $n$ blue balls ($n\geq1$), take $k+n(2k-5)$ blue balls and $2+(2n-1)(k-2)$ red balls from the bath and add them to the urn. Draw one ball from the urn at random with the uniform distribution. The probability of having a blue ball is given by $x_n$, while the probability of having a red ball is $y_n$. If we had initially drawn a red ball, then we remove all red balls in the urn and $k+n(2k-5)$ blue balls from the urn (i.e. all balls we introduced at the beginning) and start over. If we get a blue ball, then we remove all red balls and $k-1+n(2k-5)$ blue balls from the urn (so that there are $n+1$ blue balls in the urn) and start over. We follow this strategy except when the initial number of blue balls is 0, in which case there is an initial test to determine if in one step of time we add 1 blue ball to the urn. This initial test may be though of as tossing a (probably biased) coin with probability of heads equal to $ka_0$. If we get heads, then we add 1 blue ball to the urn, otherwise we repeat the coin tossing until we get heads.

The Experiment 2 (for $P_L$) is very similar but without an initial test. In this case we will have a discrete time pure death random walk on the nonnegative integers $\mathbb{Z}_{\geq0}$ (see diagram above \eqref{Pxy2}) where each state represents again the number of blue balls in the urn. If the state of the system is $n$ blue balls ($n\geq1$), take $1+n(k-3)$ blue balls and $1+(n-1)(k-2)$ red balls from the bath and add them to the urn. Draw again one ball from the urn at random. The probability of having a blue ball is given by $r_n$, while the probability of having a red ball is $s_n$. If we had initially drawn a red ball, then we remove all red balls in the urn and $1+n(k-3)$ blue balls from the urn (i.e. all balls we introduce at the beginning) and start over. If we get a blue ball, then we remove all red balls and $2+(n-1)(k-2)$ blue balls from the urn (so that there are $n-1$ blue balls in the urn) and start over. If the urn is empty we stop the experiment.

The urn model for $P$ will be the composition of both experiments, first the Experiment 1 and then the Experiment 2, while the urn model for the Darboux transformation $\widetilde P$ \eqref{DarbT} will also be the composition of both experiments but proceeds in the reversed order. There will be 4 possibles results, $n+1$, $n$ (twice) or $n-1$ blue balls in the urn. A diagram for this urn model is similar to the one given at the end of Section \ref{sec5}.

\medskip

Since $a_0\leq1/k\leq1/2, k\geq2,$ we have that the spectral measure associated with $P$ is given by only the continuous part of \eqref{spmeac}, i.e.
$$
\omega(x)=\frac{2a_0}{\pi[(4a_0-1)x+(1-2a_0)^2]}\sqrt{\D\frac{x}{1-x}},\quad x\in[0,1].
$$

In order to calculate the spectral measure $\widetilde\omega$ for the Darboux transformation $\widetilde P$ we need to compute the Geronimus transformation of $\omega$ (see \eqref{spme}). An easy computation shows
$$
\mu_{-1}=\int_0^1\frac{\omega(x)}{x}dx=\frac{1}{1-2a_0}.
$$
We exclude the case $a_0=1/2$ when $k=2$ in which case we have $y_0=0$ and $\widetilde\omega$ will be a degenerate measure. Therefore
$$
\widetilde\omega(x)=\frac{2a_0(1-ka_0)}{\pi[(4a_0-1)x+(1-2a_0)^2]\sqrt{x(1-x)}}+\frac{a_0(k-2)}{1-2a_0}\delta_0(x),\quad x\in[0,1].
$$

From the probabilistic point of view and since we have explicit expressions of the spectral measures $\omega$ and $\widetilde\omega$, we can see that $\int_0^1\frac{\omega}{1-x}=\infty$ and $\int_0^1\frac{\widetilde\omega}{1-x}=\infty$. Therefore, both random walks are always \emph{recurrent}. From the inverse of the norms of the corresponding orthogonal polynomials we see that the \emph{invariant measure} (not a distribution) is given by
$$
\bm\pi=\left(1,4a_0,4a_0,\ldots\right).
$$

\section{Random walk generated by the Jacobi polynomials}\label{sec5}

In this section we will study in detail the case of the Jacobi orthogonal polynomials. For $\alpha,\beta>-1$ define the coefficients
\begin{align}
\nonumber a_n&=\frac{(n+\beta+1)(n+1+\alpha+\beta)}{(2n+\alpha+\beta+1)(2n+2+\alpha+\beta)},\quad n\geq0,\\
\label{ttrrJ}b_n&=\frac{(n+\beta+1)(n+1)}{(2n+\alpha+\beta+1)(2n+2+\alpha+\beta)}+\frac{(n+\alpha)(n+\alpha+\beta)}{(2n+\alpha+\beta+1)(2n+\alpha+\beta)},\quad n\geq0,\\
\nonumber c_n&=\frac{n(n+\alpha)}{(2n+\alpha+\beta+1)(2n+\alpha+\beta)},\quad n\geq1.
\end{align}
Observe that all these coefficients are nonnegative, $a_0+b_0=1$ and $a_n+b_n+c_n=1, n\geq1,$ so they are the coefficients of a discrete time random walk on the nonnegative integers and depend on the state of the system. Also it is easy to see that the random walk is irreducible for the values $\alpha,\beta>-1$.

The family of polynomials generated by the three-term recursion relation
\begin{equation*}
xQ_n^{(\alpha,\beta)}(x)=a_{n}Q_{n+1}^{(\alpha,\beta)}(x)+b_nQ_n^{(\alpha,\beta)}(x)+c_{n}Q_{n-1}^{(\alpha,\beta)}(x),\quad n\geq0,
\end{equation*}
where $Q_{-1}^{(\alpha,\beta)}(x)=0$ and $Q_0^{(\alpha,\beta)}(x)=1$ is the well known family of \emph{Jacobi polynomials}, which are orthogonal with respect to the (normalized) weight
\begin{equation}\label{WW}
w(x)=\frac{\Gamma(\alpha+\beta+2)}{\Gamma(\alpha+1)\Gamma(\beta+1)}x^{\alpha}(1-x)^{\beta},\quad x\in[0,1].
\end{equation}
Notice that the Jacobi polynomials satisfy the condition 
$$
Q_n^{(\alpha,\beta)}(1)=1.
$$ 
The normalization in terms of $a_n, b_n, c_{n+1},n\geq0,$ is natural when one thinks of these polynomials (at least for some values of $\alpha,\beta$) as the spherical functions for some appropriate symmetric space, and insists that these functions take the value $1$ at the North pole of the corresponding sphere. The simplest of all cases is the one with $\alpha=\beta=0$ when one gets the Legendre polynomials and the usual two dimensional sphere sitting in ${\mathbb R}^3$.

We need to see that we can apply the stochastic UL (or LU) factorization. For that we need to apply Proposition \ref{lemcf} and see if the continued fraction $H$ in \eqref{cfUL} is convergent. In this case it is possible to check that the corresponding sequence of alternating numbers $a_0,c_1,a_1,c_2,\ldots$ is a \emph{chain sequence}. Following the notation of Remark \ref{remCS} let us call $\alpha_n, n\geq1,$ the sequence of partial numerators $a_0,c_1,a_1,c_2,\ldots$. Therefore, $\alpha_n=(1-m_{n-1})m_n$ where
$$
m_{2n}=\frac{n}{2n+\alpha+\beta+1},\quad m_{2n+1}=\frac{n+\beta+1}{2n+\alpha+\beta+2},\quad n\geq0.
$$
Since $m_0=0$, then $H$ converges to $(1+L)^{-1}$ where $L$ is given by \eqref{LL}, which in this case is a hypergeometric series. It is possible to see that
$$
L=\frac{\beta+1}{\alpha}.
$$
Since we have a chain sequence the conditions of the Proposition \ref{lemcf} hold, so we have that the stochastic UL factorization is always possible if we choose the free parameter $y_0$ in the range
\begin{equation}\label{y0r}
0\leq y_0\leq\frac{1}{1+L}=\frac{\alpha}{\alpha+\beta+1}.
\end{equation}

For the LU factorization there is no free parameter and we need to have
$$
a_0=\frac{\beta+1}{\alpha+\beta+2}\leq\widetilde H,
$$
where $\widetilde H$ is the continued fraction given by \eqref{cfLU}. In this case, if we call $\alpha_n, n\geq1,$ the sequence of partial numerators $c_1,a_1,c_2,a_2,\ldots$, then $\alpha_n$ is a \emph{chain sequence} with coefficients
$$
m_{2n}=\frac{n+\beta+1}{2n+\alpha+\beta+2},\quad m_{2n+1}=\frac{n+1}{2n+\alpha+\beta+3},\quad n\geq0.
$$
Observe now that $m_0\neq0$. Therefore we have
$$
\widetilde H=m_0+\frac{1-m_0}{1+L},
$$
where $L$ is given by \eqref{LL}. As before, it is possible to calculate $L$, which in this case it is given by $L=1/\alpha$. Therefore
$$
\widetilde H=\frac{\beta+1}{\alpha+\beta+2}+\frac{\frac{\alpha+1}{\alpha+\beta+2}}{1+1/\alpha}=\frac{\alpha+\beta+1}{\alpha+\beta+2}.
$$
We clearly see that $a_0\leq\widetilde H$ and therefore we can always perform a stochastic LU factorization, but now without a free parameter. From \eqref{ULd}, \eqref{recr} and \eqref{recx} we can see that 
\begin{align}
\label{xysrLUJ}\tilde x_n&=\frac{n+\beta+1}{2n+\alpha+\beta+2},\quad \tilde y_n=\frac{n+\alpha+1}{2n+\alpha+\beta+2},\quad n\geq0,\\
\nonumber\tilde s_n&=\frac{n+\alpha+\beta+1}{2n+\alpha+\beta+1},\quad \tilde r_n=\frac{n}{2n+\alpha+\beta+1},\quad n\geq1,\quad \tilde s_0=1.
\end{align}

\medskip

Coming back to the UL factorization, there are two cases where all coefficients $x_n,y_n,s_n,r_n$ simplify considerably, namely when $y_0$ coincides with one of the endpoints of the range \eqref{y0r}. They are
\begin{enumerate}
\item $y_0=0$. From Lemma \ref{lem22} one can check that the coefficients $x_n,y_n,s_n,r_n$ are given by 
\begin{align*}
x_n&=\frac{n+\alpha+\beta+1}{2n+\alpha+\beta+1},\quad y_n=\frac{n}{2n+\alpha+\beta+1},\quad n\geq0,\\
s_n&=\frac{n+\beta}{2n+\alpha+\beta},\quad r_n=\frac{n+\alpha}{2n+\alpha+\beta},\quad n\geq1,\quad s_0=1.
\end{align*}
\item $y_0=\frac{\alpha}{\alpha+\beta+1}$. From Lemma \ref{lem22} one can check that the coefficients $x_n,y_n,s_n,r_n$ are given by
\begin{align}
\label{xysrJ}x_n&=\frac{n+\beta+1}{2n+\alpha+\beta+1},\quad y_n=\frac{n+\alpha}{2n+\alpha+\beta+1},\quad n\geq0,\\
\nonumber s_n&=\frac{n+\alpha+\beta}{2n+\alpha+\beta},\quad r_n=\frac{n}{2n+\alpha+\beta},\quad n\geq1,\quad s_0=1.
\end{align}
These coefficients are scalar counterparts of the matrix-valued coefficients given in \cite{GPT3}, where the authors consider a special case of the UL block factorization.
\end{enumerate}

If $0<y_0<\frac{\alpha}{\alpha+\beta+1}$ then the coefficients $x_n,y_n,s_n,r_n$ are more difficult to calculate. Nevertheless it is possible to derive explicit formulas as the following:
\begin{align*}
x_n&=\frac{\gamma_n+\delta_ny_0}{(\varepsilon_n+\nu_ny_0)(2n+\alpha+\beta+1)},\quad y_n=1-x_n,\quad n\geq0,\\
s_n&=\frac{\varepsilon_n+\nu_ny_0}{(\gamma_{n-1}+\delta_{n-1}y_0)(2n+\alpha+\beta)},\quad s_n=1-r_n,\quad n\geq1,\quad s_0=1,
\end{align*}
where
\begin{align*}
\gamma_n&=(\alpha+1)_n(\alpha+\beta+2)_n,\\
\alpha\delta_n&=n!(\beta+1)_{n+1}-(\alpha+\beta+1)\gamma_n,\\
\varepsilon_n&=(\alpha+1)_n(\alpha+\beta+2)_{n-1},\\
\alpha\nu_n&=n!(\beta+1)_{n}-(\alpha+\beta+1)\varepsilon_n.
\end{align*}
These sequences also satisfy the recurrence formulas
\begin{align*}
\varepsilon_n&=(2n+\alpha+\beta)\gamma_{n-1}-(n+\beta)(n+\alpha+\beta)\varepsilon_{n-1},\\
\nu_n&=(2n+\alpha+\beta)\delta_{n-1}-(n+\beta)(n+\alpha+\beta)\nu_{n-1}.
\end{align*}
If $y_0=0$ or $y_0=\frac{\alpha}{\alpha+\beta+1}$ these formulas reduce to the simpler form introduced above. There is a more convenient way of writing the coefficients $x_n,y_n,s_n,r_n$ if we make the substitution $y_0=\frac{\alpha}{\alpha+\beta+1}h_0$ (in which case $0\leq h_0\leq 1$). Indeed,
\begin{align}
\label{xysrJE}x_n&=\frac{n+\alpha+\beta+1}{2n+\alpha+\beta+1}\cdot\frac{h_0n!(\beta+1)_n(n+\beta+1)+(1-h_0)(\alpha+1)_n(\alpha+\beta+1)_{n+1}}{h_0n!(\beta+1)_n(n+\alpha+\beta+1)+(1-h_0)(\alpha+1)_n(\alpha+\beta+1)_{n+1}},\\
\nonumber y_n&=\frac{n+\alpha}{2n+\alpha+\beta+1}\cdot\frac{h_0n!(\beta+1)_n(n+\alpha)+(1-h_0)(\alpha+1)_{n}(\alpha+\beta+1)_{n}n}{h_0n!(\beta+1)_n(n+\alpha)+(1-h_0)(\alpha+1)_{n}(\alpha+\beta+1)_{n}(n+\alpha)},\\
\nonumber s_n&=\frac{n+\alpha+\beta}{2n+\alpha+\beta}\cdot\frac{h_0(n-1)!(\beta+1)_{n}(n+\alpha+\beta)+(1-h_0)(\alpha+1)_{n-1}(\alpha+\beta+1)_{n}(n+\beta)}{h_0(n-1)!(\beta+1)_n(n+\alpha+\beta)+(1-h_0)(\alpha+1)_{n-1}(\alpha+\beta+1)_{n}(n+\alpha+\beta)},\\
\nonumber r_n&=\frac{n+\alpha}{2n+\alpha+\beta}\cdot\frac{h_0n!(\beta+1)_n+(1-h_0)(\alpha+1)_{n}(\alpha+\beta+1)_{n}}{h_0(n-1)!(n+\alpha)(\beta+1)_n+(1-h_0)(\alpha+1)_{n}(\alpha+\beta+1)_{n}}.
\end{align}
From these formulas we clearly see that each coefficient $x_n,y_n,s_n,r_n$ is the multiplication of two positive numbers less than 1. We will give probabilistic implementations of these coefficients in the next subsection.

\medskip

Finally we explore the spectral measure associated with the Darboux transformation \eqref{DarbT}. We need to apply the Geronimus transformation \eqref{spme}. In this case it is easy to see from \eqref{WW} that
$$
\mu_{-1}=\int_0^1\frac{w(x)}{x}dx=\frac{\alpha+\beta+1}{\alpha}.
$$
Therefore, following \eqref{spme}, we have that
\begin{equation*}
\widetilde{w}(x)=y_0\frac{\Gamma(\alpha+\beta+2)}{\Gamma(\alpha+1)\Gamma(\beta+1)}x^{\alpha-1}(1-x)^{\beta}+\left(1-y_0\frac{\alpha+\beta+1}{\alpha}\right)\delta_0(x),\quad x\in[0,1].
\end{equation*}
This measure is integrable as long as $\alpha>0$ and $\beta>-1$. We see that if $y_0$ is in the range \eqref{y0r} then the mass at 0 is always nonnegative, and vanishes if $y_0=\frac{\alpha}{\alpha+\beta+1}$. The case  $y_0=0$ was treated in Remark \ref{remy02} and in this case the spectral measure is just the Dirac delta $\delta_0(x)$ (degenerate).

Finally, for the LU decomposition, the associated (normalized) weight is given by the Christoffel transform of $w$, i.e.
$$
\widehat w(x)=\frac{\Gamma(\alpha+\beta+3)}{\Gamma(\alpha+2)\Gamma(\beta+1)}x^{\alpha+1}(1-x)^{\beta},\quad x\in[0,1].
$$

In order to study recurrence we have to see the behavior of the integral $\int_0^1\frac{w(x)}{1-x}dx$. But we can see that this behavior only depends on the parameter $\beta$. If $-1<\beta\leq0$ then the integral is $\infty$, so the process will always be \emph{recurrent}. Otherwise, for $\beta>0$, the process will be \emph{transient}.

Again the inverse of the norms of the Jacobi orthogonal polynomials gives the explicit expression of the components of the invariant measure (not a distribution in any case). Indeed we have
$$
\|Q_n^{(\alpha,\beta)}\|^2_{w}=\frac{n!(\alpha+1)_n}{(\beta+1)_n(\alpha+\beta+2)_{n-1}(2n+\alpha+\beta+1)},
$$
Therefore
$$
\bm\pi=\left(1,\frac{(\beta+1)(\alpha+\beta+3)}{\alpha+1},\frac{(\beta+1)_2(\alpha+\beta+2)(\alpha+\beta+5)}{2(\alpha+2)_2},\cdots\right).
$$

\subsection{An urn model for the Jacobi polynomials}

We now give an urn model associated with the Jacobi polynomials. We will use first the simpler UL decomposition with coefficients given by \eqref{xysrJ} and then make a few comments about the general case \eqref{xysrJE}. In the latter case, as in the previous section, we have a family of urn model experiments, which combined give the same model for the Jacobi polynomials. A similar urn model can be derived from the LU decomposition \eqref{xysrLUJ}, but now without a free parameter.

In \cite{G3} one finds what is probably the first urn model going along with the Jacobi polynomials. This is a rather contrived model when compared to more familiar ones such as those of Ehrenfest and Bernoulli-Laplace. These can be found in W. Feller's classical book, see \cite{Fe}. It turns out that these celebrated models are related to the Krawtchouk and Hahn orthogonal polynomials, in the same way that our models are related to the Jacobi polynomials. Here we give a slightly less elaborate urn model based on the composition of two easier urn experiments given by the UL factorization \eqref{xysrJ}.

From now on, it will be assumed that the parameters $\alpha$ and $\beta$ are nonnegative integers. Consider the discrete time random walk on the nonnegative integers $\mathbb{Z}_{\geq0}$ whose one step transition probability matrix $P$ coincides with the one that gives the three-term recursion relation given in \eqref{ttrrJ}. Consider the UL factorization $P=P_UP_L$ \eqref{Pxy} with coefficients $x_n,y_n,s_n,r_n$ given by \eqref{xysrJ}. Each one of these matrices $P_U$ and $P_L$ will represent an experiment in terms of an urn model, which we call Experiment 1 and Experiment 2, respectively. At times $t=0,1,2,\ldots$ an urn contains $n$ blue balls and this determines the state of our random walk on ${\mathbb Z}_{\geq 0}$ at that time. Each urn for both experiments sits in a bath consisting of an infinite number of blue and red balls.

The Experiment 1 (for $P_U$) consists of a discrete time pure birth random walk on the nonnegative integers ${\mathbb Z}_{\geq 0}$ (see diagram below \eqref{ULd}). If the state of the system is $n$ blue balls ($n\geq0$), take $\beta+1$ blue balls and $n+\alpha$ red balls from the bath and add them to the urn. Draw one ball from the urn at random with the uniform distribution. The probability of having a blue ball is given by $x_n$, while the probability of having a red ball is $y_n$. If we had initially drawn a red ball, then we remove all red balls in the urn and $\beta+1$ blue balls from the urn (i.e. all balls we introduce in this first step) and start over. If we get a blue ball, then we remove all red balls and $\beta$ blue balls from the urn (so that there are $n+1$ blue balls in the urn) and start over.

The Experiment 2 (for $P_L$) consists of a discrete time pure death random walk on the nonnegative integers ${\mathbb Z}_{\geq 0}$ (see diagram above \eqref{Pxy2}). If the state of the system is $n$ blue balls ($n\geq0$), take $n+\alpha+\beta$ red balls from the bath and add them to the urn. Draw again one ball from the urn at random. The probability of having a blue ball is given by $r_n$, while the probability of having a red ball is $s_n$. If we had initially drawn a red ball, then we remove all red balls in the urn and start over. If we get a blue ball, then we remove that blue ball and all red balls from the urn (so that there are $n-1$ blue balls in the urn) and start over. If the urn is empty we stop the experiment. Observe that if we have 0 blue balls in the urn the experiment will not change from that moment on.

As in the previous section, the urn model for $P$ will be the composition of Experiment 1 and then Experiment 2, while the urn model for the Darboux transformation $\widetilde P$ \eqref{DarbT} proceeds in the reversed order. If we perform first Experiment 1 we will end up with an urn with either $n$ (if we draw a red ball) or $n+1$ (if we draw a blue ball) blue balls. Now we perform Experiment 2 with $n$ or $n+1$ blue balls, in which case we may have either $n-1$ (if we draw a blue ball) or $n$ (if we draw a red ball) blue balls, while for the $n+1$ case we may have either $n$ (if we draw a blue ball) or $n+1$ (if we draw a red ball) blue balls. The combination of probabilities of these four cases gives the coefficients of the three-term recurrence relation \eqref{ttrrJ} (see also \eqref{ULd}) for $P$ for the Jacobi polynomials (see diagram below).
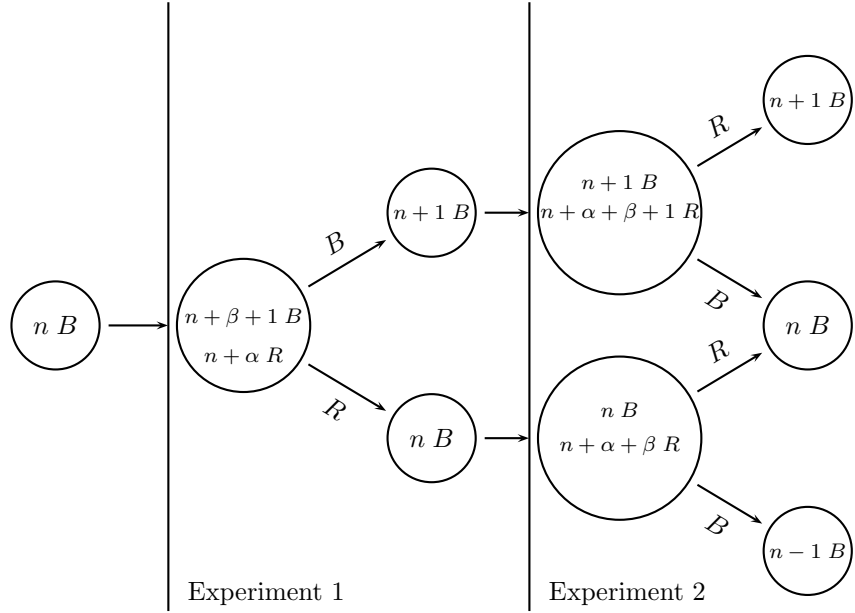
\vspace{.5cm}
\begin{center}
$$
\cnode(-5,-3){.6cm}{0}
\cnode(-2.5,-3){0.9cm}{1}\cnode(0,-1.5){.6cm}{2}\cnode(0,-4.5){.6cm}{3}\cnode(2.5,-1.5){1.1cm}{4}\cnode(2.5,-4.5){1.1cm}{5}
\cnode(5,0){.6cm}{6}\cnode(5,-3){.6cm}{7}\cnode(5,-6){.6cm}{8}
\psset{nodesep=3pt,nrot=:U,arrows=->}
\ncline{->}{0}{1}
\ncline{->}{1}{2}
\naput*{B}
\ncline{->}{1}{3}
\nbput*{R}
\ncline{->}{2}{4}
\ncline{->}{3}{5}
\ncline{->}{4}{6}
\naput*{R}
\ncline{->}{4}{7}
\nbput*{B}
\ncline{->}{5}{7}
\naput*{R}
\ncline{->}{5}{8}
\nbput*{B}
\psset{labelsep=-4.75ex}\nput{90}{0}{n\; B}
\uput[u](-2.5,-2.3){\mbox{{\footnotesize $n+\beta+1\; B$}}}
\uput[u](-2.5,-2.8){\mbox{{\footnotesize $n+\alpha\; R$}}}
\psset{labelsep=-4.75ex}\nput{90}{2}{\mbox{{\footnotesize $n+1  \;B$}}}
\psset{labelsep=-4.75ex}\nput{90}{3}{n\; B}
\uput[u](2.5,-0.5){\mbox{{\footnotesize $n+1\; B$}}}
\uput[u](2.5,-0.9){\mbox{{\footnotesize $n+\alpha+\beta+1\; R$}}}
\uput[u](2.5,-3.5){\mbox{{\footnotesize $n\; B$}}}
\uput[u](2.5,-4.0){\mbox{{\footnotesize $n+\alpha+\beta\; R$}}}
\psset{labelsep=-4.75ex}\nput{90}{6}{\mbox{{\footnotesize $n+1  \;B$}}}
\psset{labelsep=-4.75ex}\nput{90}{7}{n\; B}
\psset{labelsep=-4.75ex}\nput{90}{8}{\mbox{{\footnotesize $n-1  \;B$}}}
\qline(-3.5,1.3)(-3.5,-6.8)
\qline(1.3,1.3)(1.3,-6.8)
\uput[u](-2.2,-6.0){\mbox{Experiment 1}}
\uput[u](2.6,-6.0){\mbox{Experiment 2}}
$$
\end{center}
\vspace{7cm}

For the general coefficients $x_n,y_n,s_n,r_n$ in \eqref{xysrJE} we obtain a family of Experiments 1 and 2, depending on $0\leq h_0\leq1$. Now both experiments are more complicated, but we observe that each probability $x_n,y_n,s_n,r_n$ is the multiplication of two positive numbers less than 1. Than means that for each experiment we will have to perform another experiment (different for each experiment) which combined gives the probabilities $x_n,y_n,s_n,r_n$. All these experiments can be seen as urn models with blue and red balls very similar to the ones we have already shown. There are certain choices of the parameters where all coefficients simplify considerably. For instance, for $\alpha=\beta=0$, we have
$$
x_n=\frac{n+1}{2n+1},\quad y_n=\frac{n}{2n+1},\quad n\geq0,\quad s_n=r_n=\frac{1}{2},\quad n\geq1,\quad s_0=1.
$$
Another (more complicated) example is for $\alpha=1, \beta=0$ and $h_0=1/2$:
\begin{align*}
x_n&=\frac{1+(n+1)(n+2)}{2[1+(n+1)^2]},\quad y_n=\frac{1+n(n+1)}{2[1+(n+1)^2]},\quad n\geq0,\\
s_n&=\frac{(n+1)(n^2+1)}{(2n+1)(1+n(n+1))},\quad r_n=\frac{n(1+(1+n)^2)}{(2n+1)(1+n(n+1))},\quad n\geq1,\quad s_0=1.
\end{align*}

Different (and more complicated) urn models for matrix-valued generalizations of Jacobi polynomials, as well as other probabilistic models in terms of Young diagrams, can be found in \cite{GPT3}. The model above can be seen as an application of rather sophisticated ideas from group representation theory in \cite{GPT3} to a much more classical setup.


\begin{thebibliography}{99}



\bibitem{CG} M.M. Castro and F.A. Gr\"unbaum, \emph{On a seminal paper by Karlin and McGregor},  SIGMA \textbf{9} (2013), 020, 11 pages.

\bibitem{Ch} T. S. Chihara, \emph{An Introduction to Orthogonal Polynomials}, Gordon and Breach, NY, 1978. 

\bibitem{DRSZ} H. Dette, B. Reuther, W. Studden and M. Zygmunt, {\em Matrix measures and random walks with a block tridiagonal transition matrix}, SIAM J. Matrix Anal. Applic. \textbf{29}, No. 1 (2006), 117--142.

\bibitem{Fe}\textrm{W. Feller},
\textit{An introduction to probability theory and its applications}
(vol. 1), John Wiley \& Sons Inc, 1968.

\bibitem{Gr1} W.K. Grassmann, \emph{Means and variances of time averages in Markovian environments}, Eur. J. Oper. Res. \textbf{31} (1987), 132--139.

\bibitem{Gr2} W.K. Grassmann, \emph{Means and variances in Markov reward systems}, in Linear Algebra, Markov Chains and Queueing Models, ed. C.D. Meyer and R.J. Plemmons. Springer-Verlag, NY, 1993.

\bibitem{G2}F.A. Grünbaum, \textit{Random walks and orthogonal polynomials: some challenges}, Probability, Geometry and Integrable Systems, MSRI Publication, volumen \textbf{55}, 2007.

\bibitem{G3} F.A. Gr\"unbaum, \emph{An urn model associated with Jacobi polynomials}, Commun. Applied Math. Comput. Sciences \textbf{5} (2010), no. 1, 55--63.

\bibitem{GH}  F.A. Gr\"unbaum and L. Haine,
\emph{Orthogonal polynomials satisfying differential equations: the role of the Darboux transformation}, in: D. Levi, L. Vinet, P. Winternitz (Eds.), Symmetries an Integrability of Differential Equations, CRM Proc. Lecture Notes, vol. 9, Amer. Math. Soc. Providence, RI, 1996, 143--154.

\bibitem{GHH} F.A. Grünbaum, L. Haine and E. Horozov,
{\em Some functions that generalize the Krall-Laguerre polynomials}, J. Comp. Appl. Math. {\bf 106} (1999),
271--297.

\bibitem{GPT3} F.A. Gr\"unbaum, I. Pacharoni and J.A. Tirao, {\em Two stochastic models of a random walk in the U($n$)-spherical duals of U($n+1$)}, Ann. Mat. Pura Appl. \textbf{192} (2013), no. 3, 447--473. 

\bibitem{Hey} D.P. Heyman, \emph{A decomposition theorem for infinite stochastic matrices}, J. Appl. Prob. {\bf 32} (1995), 893--903.


\bibitem{KMc2}  S. Karlin and J. McGregor, {\em
The differential equations of birth and death processes, and the Stieltjes moment problem}, Trans. Amer. Math. Soc. \textbf{85} (1957), 489--546.

\bibitem{KMc3}  S. Karlin and J. McGregor, {\em The classification of birth-and-death processes}, Trans. Amer. Math. Soc. \textbf{86} (1957), 366--400.

\bibitem{KMc6}  S. Karlin and J. McGregor, {\em Random walks}, IIlinois J. Math. \textbf{3} (1959), 66--81.

\bibitem{MS} V.B. Matveev and M.A. Salle, \emph{Differential-difference evolution equations II: Darboux transformation for the Toda lattice}, Lett. Math. Phys. \textbf{3} (1979) 425--429.

\bibitem{SZ} V. Spiridonov and A. Zhedanov, \emph{Self-similarity, Spectral Transformations and Orthogonal and Biorthogonal Polynomials in Self-Similar Systems}, V.B. Priezzhev and V.P.Spiridonov Editors. Proc. International Workshop JINR. Dubna 1999. 349--361.

\bibitem{Vig}  V. Vigon, {\em LU factorization versus Wiener-Hopf factorization for Markov chains}, Acta Appl. Math. \textbf{128} (2013), 1--37.

\bibitem{Wa} H.S. Wall, \emph{Analytic theory of continued fractions}, D. van Nostrand Co., N.Y., 1948.

\bibitem{Yo} G.J. Yoon, \emph{Darboux transforms and orthogonal polynomials}, Bull. Korean Math. Soc. \textbf{39}  (2002), 359--376.

\bibitem{Ze} A. Zhedanov, \emph{Rational Spectral Transformations and Orthogonal Polynomials}, J. of Comp. Appl. Math. \textbf{85} (1997), 67--86.

\end{thebibliography}
\end{document}